\documentclass[11pt]{amsart}
\usepackage{amsaddr}
\usepackage{amssymb}
\usepackage{amsmath}
\usepackage{color}
\usepackage{pdfsync}
\usepackage{tikz}
\usepackage[hypertexnames=false,backref=page]{hyperref}
\usepackage[capitalize]{cleveref}
\usetikzlibrary{calc}

\title{On Color Preserving Automorphisms of Cayley Graphs of Odd Square-free Order}
\author{Edward Dobson}
\address{Department of Mathematics and Statistics, Mississippi State University, Mississippi State, MS 39762, USA and \\ University of Primorska, UP IAM, Muzejski trg 2, SI-6000 Koper, Slovenia}
\email{dobson@math.msstate.edu}
\email{ademir.hujdurovic@upr.si}
\author{Ademir Hujdurovi\' c, Klavdija Kutnar}
\address{University of Primorska, UP IAM, Muzejski trg 2, SI-6000 Koper, Slovenia and\\ University of Primorska, UP FAMNIT, Glagolja\v{s}ka 8, SI-6000 Koper, Slovenia}
\email{klavdija.kutnar@upr.si}
\author{Joy Morris}
\address{Department of Mathematics and Computer Science, University of Lethbridge,
Lethbridge,\\ Alberta, T1K 3M4, Canada}
\email{joy.morris@uleth.ca}

\numberwithin{equation}{section}

\newtheorem{thrm}[equation]{Theorem}

\newtheorem{lem}[equation]{Lemma}

\newtheorem{prop}[equation]{Proposition}

\theoremstyle{definition}
\newtheorem{defin}[equation]{Definition}

\newtheorem{hey}[equation]{Remark}

 \newcounter{case}

 \renewcommand{\thecase}{\arabic{case}}

\newcounter{subcase}

 \renewcommand{\thesubcase}{\alph{subcase}}

\newcommand{\boxprod}{\mathbin{\square}}
\usepackage{fullpage}

\def\tl{\triangleleft}
\def\AGL{{\rm AGL}}

\def\Ker{{\rm Ker}}

\def\fix{{\rm fix}}

\def\Aut{{\rm Aut}}
\def\Stab{{\rm Stab}}
\def\la{\langle}
\def\ra{\rangle}
\def\Z{{\mathbb Z}}
\def\soc{{\rm soc}}
\def\PSL{{\rm PSL}}
\def\PGammaL{{\rm P\Gamma L}}

\def\PSL{{\rm PSL}}

\def\PGL{{\rm PGL}}

\def\Cay{{\rm Cay}}

\begin{document}

\pagestyle{plain}

\baselineskip = 1.3\normalbaselineskip

\maketitle

\begin{abstract}
An automorphism $\alpha$ of a Cayley graph $\Cay(G,S)$ of a group $G$ with connection set $S$ is color-preserving if $\alpha(g,gs) = (h,hs)$ or $(h,hs^{-1})$ for every edge $(g,gs)\in E(\Cay(G,S))$.  If every color-preserving automorphism of $\Cay(G,S)$ is also affine, then $\Cay(G,S)$ is a CCA (Cayley color automorphism) graph.  If every Cayley graph $\Cay(G,S)$ is a CCA graph, then $G$ is a CCA group.  Hujdurovi\' c, Kutnar, D.W. Morris, and J. Morris have shown that every non-CCA group $G$ contains a section isomorphic to the nonabelian group $F_{21}$ of order $21$.  We first show that there is a unique non-CCA Cayley graph $\Gamma$ of $F_{21}$.  We then show that if $\Cay(G,S)$ is a non-CCA graph of a group $G$ of odd square-free order, then $G = H\times F_{21}$ for some CCA group $H$, and $\Cay(G,S) = \Cay(G,T)\boxprod\Gamma$.
\end{abstract}

\section{Introduction and preliminaries}

We consider Cayley digraphs $\Cay(G,S)$ of a group $G$ with connection set $S$ whose arcs $(g,gs)$ are colored with the color $s$, for $s\in S$.  It has been known since at least the early 1970s \cite[Theorem 4-8]{White1973} that any color-preserving automorphism of $\Cay(G,S)$ can only be an automorphism of $\Cay(G,S)$ induced by left translation by an element $k\in G$.  That is, the only color-preserving automorphisms of $\Cay(G,S)$ are of the form $x\mapsto kx$ for some $k\in G$.  The corresponding question for Cayley graphs $\Cay(G,S)$ was not considered until recently \cite{HujdurovicKMM2016}.  Note that the essential difference between the graph and digraph problem is that for a graph we insist that the connection set $S$ satisfies $S = S^{-1}$ and that we insist that any pair of arcs between two vertices $u$ and $v$ are colored with the {\it same} color (whereas in the digraph case we insist that the arcs between two vertices $u$ and $v$ are colored with {\it different} colors).  It is easy to see that there can be additional color-preserving automorphisms of Cayley graphs.  For example, if $G$ is an abelian group (written multiplicatively), then the map $x\mapsto x^{-1}$ is a color-preserving automorphism of any Cayley graph $\Cay(G,S)$.  In this case, $x\mapsto x^{-1}$ is also a group automorphism of $G$, so it makes sense to ask if there are color-preserving automorphisms of $\Cay(G,S)$ that are not group automorphisms of $G$ or translations of $G$; that is, if there are color-preserving automorphisms of $\Cay(G,S)$ which are not {\bf affine}.  Hujdurovi\' c, Kutnar, D.W. Morris, and J. Morris showed \cite{HujdurovicKMM2016} that non-affine color-preserving automorphisms of $\Cay(G,S)$ do exist for some groups $G$, and when $G$ is of odd order found fairly restrictive conditions for such a group $G$ to exist.  In this paper we continue this program, and show that groups $G$ of odd square-free order $n$ for which there exists a Cayley graph $\Cay(G,S)$ with a color-preserving automorphism that is not affine have the form $H\times F_{21}$, where $H$ is group of order $n/21$ and $F_{21}$ is the nonabelian group of order $21$.  Additionally, we show that $\Cay(G,S)$ is the Cartesian product $\Cay(H,T)\boxprod\Gamma$, where $T\subset H$ satisfies $T = T^{-1}$ and $\Gamma$ is the unique non-CCA Cayley graph of $F_{21}$.

Throughout this paper, all groups and graphs are finite.

\begin{defin}
Let $G$ be a group and $S\subset G$ such that $1_G\not\in S$. Define a {\bf Cayley digraph of
$G$}, denoted $\Cay(G,S)$, to be the digraph with vertex set $V(\Cay(G,S)) = G$ and arc set $A(\Cay(G,S)) = \{(g,gs):g\in G, s\in S\}$.  We call $S$ the {\bf connection set of $\Cay(G,S)$}.
\end{defin}

If $S = S^{-1}$ then $\Cay(G,S)$ is a graph, and we call the arc set the {\bf edge set}.  For $g\in G$, define $g_L:G\mapsto G$ by $g_L(x) = gx$.  It is straightforward to verify that $G_L = \{g_L:g\in G\}$ is a group isomorphic to $G$.  The group $G_L$ is the {\bf left regular representation of $G$} and it is an easy exercise to show that $G_L\le\Aut(\Cay(G,S))$.

\begin{defin}\label{def:coloring}
Let $S\subset G$ such that $S = S^{-1}$ and to each pair $s,s^{-1}\in S$, assign a unique color $c(s)=c(s')$, so that $c(s')=c(s)$ implies $s' \in \{s,s^{-1}\}$.  Let $S' = \{c(s):s\in S\}$.  Consider a Cayley graph $\Cay(G,S)$ in which each edge $(g,gs)$ is colored with $c(s)\in S'$, and $E_{c(s)}$ be the set of all edges of $\Cay(G,S)$ that are colored with the color $c(s)\in S'$.  An automorphism $\alpha$ of $\Cay(G,S)$ is a {\bf color-preserving} automorphism if $\alpha(E_{c(s)}) = E_{c(s)}$ for each ${c(s)}\in S'$.  The set of all color-preserving automorphisms of $\Cay(G,S)$ is a group denoted ${\mathcal A}^o$.
\end{defin}

Clearly $G_L\le {\mathcal A}^o$, and if $\alpha\in\Aut(G)$, then $\alpha\in {\mathcal A}^o$ if and only if $\alpha(\{s,s^{-1}\}) = \{s,s^{-1}\}$ for every $s\in S$.

The group $G_L\cdot\Aut(G)\le S_G$ (where $S_G$ denotes the symmetric group on $G$), is the normalizer in $S_G$ of $G_L$ by \cite[Corollary 4.2B]{DixonM1996}, and an element of $G_L\cdot\Aut(G)$ is called {\bf affine}.  As mentioned earlier, automorphism groups of Cayley graphs of abelian groups always contain an affine automorphism that is not in $G_L$, namely $x\mapsto x^{-1}$. Additionally, as an element $\alpha\in\Aut(G)$ is contained in $\Aut(\Cay(G,S))$ and is color-preserving if and only if $\alpha(\{s,s^{-1}\}) = \{s,s^{-1}\}$ for every $s\in S$, given a Cayley graph $\Cay(G,S)$ is it straightforward to compute the subgroup of $G_L\cdot\Aut(G)$ that is color-preserving.  Thus, the {real challenge lies in determining whether or not a given Cayley graph has color-preserving automorphisms that are not affine.

\begin{defin}
We say a Cayley graph $\Cay(G,S)$ of a group $G$ is a CCA-graph if every color-preserving automorphism of $\Cay(G,S)$ is affine.  A group $G$ is a CCA-group if and only if every connected Cayley graph of $G$ is a CCA-graph.
\end{defin}

In \cite[Theorem 6.8]{HujdurovicKMM2016}, the following constraints on groups of odd order that are not CCA were obtained.  Before stating this result, we need another definition.

\begin{defin}
Let $G$ be a group. For any subgroups $H$ and $K$ of $G$, such that $K\tl H$, the quotient $H/K$ is said to be a {\bf section} of $G$.
\end{defin}

\begin{thrm}\cite[Theorem 6.8]{HujdurovicKMM2016}\label{bigpreviousresult}
Any non-CCA group of odd order has a section that is isomorphic to either:
\begin{enumerate}
\item a semi-wreathed product $A\wr_α \Z_n$ (see \cite[Example 6.5]{HujdurovicKMM2016} for the appropriate definitions), where $A$ is a nontrivial, elementary abelian group (of odd order) and $n > 1$, or
\item the (unique) nonabelian group of order $21$.
\end{enumerate}
\end{thrm}

The groups in (1) do not have square-free order, so for the odd square-free integers $n$ under consideration in this paper, the preceding result says that they are CCA-groups unless they contain a section isomorphic to the nonabelian group $F_{21}$ of order $21$.  We now turn to improving this result, and begin with some preliminary results and the definitions needed for them.

\begin{defin}
Let $G$ be a transitive permutation group with invariant partition ${\mathcal B}$.  By $G/{\mathcal B}$, we mean the subgroup of
$S_{\mathcal B}$ induced by the action of $G$ on ${\mathcal B}$, and by $\fix_G({\mathcal B})$ the kernel of this action.
Thus $G/{\mathcal B} = \{g/{\mathcal B}:g\in G\}$ where $g/{\mathcal B}(B_1) = B_2$ if and only if $g(B_1) = B_2$, $B_1,B_2\in{\mathcal B}$, and $\fix_G({\mathcal B}) = \{g\in G:g(B) = B{\rm\ for\ all\ }B\in{\mathcal B}\}$.  It is often the case that $G$ will have invariant partitions ${\mathcal B}$ and ${\mathcal C}$, and every block of ${\mathcal C}$ is a union of blocks of ${\mathcal B}$.  In this case, we write ${\mathcal B}\preceq{\mathcal C}$.  Then $G/{\mathcal B}$ admits an invariant partition with blocks consisting of those blocks of ${\mathcal B}$ contained in a block of ${\mathcal C}$.  For $C\in{\mathcal C}$, we let $C/{\mathcal B}$ be the blocks of ${\mathcal B}$ that are contained in $C$, and write ${\mathcal C}/{\mathcal B}$ for $\{C/{\mathcal B}:C\in{\mathcal C}\}$.
\end{defin}

\begin{prop}\label{calAoimprimitive}
Let $G$ be a group and $1_G\neq s \in S\subset G$.  Then the left cosets of $\la s\ra$ in $G$ form an invariant partition of the color-preserving group of automorphisms ${\mathcal A}^o$ of $\Cay(G,S)$.  Consequently, if there is some $s \in S$ such that $\langle s \rangle\neq G, \{1_G\}$, then ${\mathcal A}^o$ is imprimitive.
\end{prop}

\begin{proof}
The Cayley graph $\Cay(G, \{s^{\pm 1}\})$ has as its connected components the left cosets of $\la s\ra$ in $G$.  If $a\in{\mathcal A}^o$, then $a(\Cay(G,\{s^{\pm 1}\})) = \Cay(G,\{s^{\pm 1}\})$ and $a$ certainly maps each connected component of $\Cay(G,\{s^{\pm 1}\})$ to some connected component of $\Cay(G,\{s^{\pm 1}\})$. So the left cosets of $\la s\ra$ form an invariant partition ${\mathcal B}$ of ${\mathcal A}^o$.  Finally, the ${\mathcal A}^o$-invariant partition ${\mathcal B}$ is nontrivial, making $\mathcal A^o$ imprimitive, if and only if $s\neq 1_G$ does not generate $G$.
\end{proof}

\begin{defin}
Let $G\le S_X$ be a transitive permutation group, and ${\mathcal O}_0,\ldots,{\mathcal O}_r$ the orbits of $G$ acting on $X\times X$.  Assume that ${\mathcal O}_0$ is the diagonal orbit $\{(x,x):x\in X\}$.  Define digraphs $\Gamma_1,\ldots,\Gamma_r$ by $V(\Gamma) = X$ and $A(\Gamma) = {\mathcal O}_i$.  The graphs $\Gamma_1,\ldots,\Gamma_r$ are the orbital digraphs of $G$.  Define the {\bf $2$-closure of $G$}, denoted by $G^{(2)}$, as $\cap_{i=1}^r\Aut(\Gamma_i)$.  We say $G$ is {\bf $2$-closed} if $G^{(2)} = G$.
\end{defin}

Observe that if the arcs of each $\Gamma_i$ are colored with color $i$, $1\le i\le r$, then $G^{(2)}$ is the automorphism group of the resulting color digraph. Therefore, we can equivalently define that $G$ is $2$-closed if and only if $G$ is the automorphism group of a colour (di)graph. This is the definition we will use in our next proof.

\begin{lem}\label{color preserving 2-closed}
For a Cayley graph $\Cay(G,S)$, the group ${\mathcal A}^o$ is $2$-closed.
\end{lem}

\begin{proof}
We show that ${\mathcal A}^o$ is the automorphism group of a Cayley color graph.  This is done in the obvious way.  Namely, we edge-color $\Cay(G,S)$ with the natural edge-coloring described in \cref{def:coloring}.  It is then clear that ${\mathcal A}^o$ is the automorphism group of the resulting Cayley color graph.
\end{proof}

\section{CCA graphs of $F_{21}$ and complete CCA graphs}

In this section, we will show that $\Cay(F_{21},\{a^{\pm 1},(ax)^{\pm 1}\})$ is the unique non-CCA graph of $F_{21}$, where $F_{21} = \la a,x | a^3 = x^7 = e, a^{-1}xa = x^2\ra$ is the nonabelian group of order $21$.  This example was first given in \cite[Example 2.3]{HujdurovicKMM2016}, and is drawn in \cref{F21nCCA} (note that $(ax)^{\pm 1} = \{x^4a,x^6a^2\}$).  Edges corresponding to colors $a^{\pm 1}$ are in black, while edges corresponding to $(ax)^{\pm 1}$ are in red.  We will make use of \cite[Theorem 3.2]{Dobson2006a}, and observe that while this result is stated for graphs, the proofs hold for digraphs $\Gamma$ and $2$-closed groups $G$ provided that $\Aut(\Gamma)$ (or $G$) has a nontrivial invariant partition formed by the orbits of a normal subgroup (but the result does not hold for digraphs or $2$-closed groups if this condition is not satisfied).

In order to show that $\Cay(F_{21},\{a^{\pm 1},(ax)^{\pm 1}\})$ is the unique non-CCA graph of $F_{21}$, we require a separate argument to show that the complete Cayley graph on $F_{21}$ is CCA. Since this argument generalizes fairly straightforwardly to any complete Cayley graph on a group that is not a Hamiltonian 2-group, we present the generalization here. We begin with a couple of results we will need repeatedly in the proof.

\begin{lem}\label{inverters}
Let $\Gamma=\Cay(G,G\setminus\{1_G\})$ be a complete graph, viewed as a Cayley graph on a group $G$, and let $\varphi$ be a color-preserving automorphism that fixes $1_G$. If $g, x \in G$ with $\varphi(x)=x^{-1}\neq x$ and $\varphi(g)=g$, then $x^{-1}gx=g^{-1}$. Furthermore, if $\varphi$ does not invert every element of $G$, then $|x|=4$.
\end{lem}

\begin{proof}
Let $h=g^{-1}x$, so that $x=gh$. Now, $x \sim g$ via an edge of color $c(h)$ (where $x\sim g$ means there is an edge between $x$ and $g$), so $\varphi(x)$ must be adjacent to $\varphi(g)=g$ via an edge of color $c(h)$, meaning $\varphi(x) \in \{gh,gh^{-1}\}$. Thus, $(gh)^{-1} \in \{gh,gh^{-1}\}$. If $|h|=2$ so that $gh^{-1}=gh$ or $(gh)^{-1}=gh$, then $x^{-1}=\varphi(x)=x$, contradicting our choice of $x$. So we must have $|h|>2$ and $(gh)^{-1}=gh^{-1}$, i.e., $gh^{-1}gh=1_G$, so $h^{-1}gh=g^{-1}$. Hence $x^{-1}gx=h^{-1}gh=g^{-1}$.

Observe that for every $a \in G$, either $a=1_G$ or $a \sim 1_G$. Since $\varphi$ is color-preserving, this means $\varphi(a)\in \{a, a^{-1}\}$.
So if $\varphi$ does not invert every element of $G$ then we may assume without loss of generality that it fixes some $g$ with $|g|>2$. Let $y=gh^2$. Since $y\neq g$ and $y \sim x$ via an edge of color $c(h)$, we must have $\varphi(y)=gh^{-2}$. If $\varphi(y)=y^{-1}$, then $gh^{-2}gh^2=1_G$, but this contradicts $h^{-1}gh=g^{-1}$ since $|g|>2$. So we must have $|h|=4$ and $\varphi(y)=y$. Now since $h^{-1}gh=g^{-1}$ and $x=gh$, we see that $x^2=h^2$, so $|x|=4$. Also, $x^{-1}gx=h^{-1}gh=g^{-1}$. Thus, for any $g, x \in G$ with $\varphi(g)=g$ and $\varphi(x)=x^{-1}$, we have $x$ inverts $g$ and $|x|=4$.
\end{proof}

As was mentioned above, Hamiltonian $2$-groups play an important role in this statement. We give here their definition and some key facts.

\begin{defin}
A \textbf{Hamiltonian 2-group} is a nonabelian 2-group, all of whose subgroups are normal. It was proven by Dedekind in the finite case, and extended by Baer to the infinite case (and is now well-known), that Hamiltonian $2$-groups have the form $Q_8 \times \Z_2^{n}$ for some non-negative integer $n$.
\end{defin}

\begin{thrm}\label{complete-CCA}
Let $\Gamma=\Cay(G,G\setminus\{1_G\})$ be a complete graph, viewed as a Cayley graph on a group $G$. Then $\Gamma$ is a CCA graph if and only if $G$ is not a Hamiltonian $2$-group.
\end{thrm}

\begin{proof}
First suppose that $G$ is a Hamiltonian $2$-group. Define $\varphi$ by $\varphi(x)=x^{-1}$ for every $x \in G$. Since Hamiltonian $2$-groups are nonabelian, $\varphi$ is not a group automorphism of $G$. To show that $\varphi$ is color-preserving, let $x, y \in G$ with $y=xh$. Then $\varphi(y)=y^{-1}=h^{-1}x^{-1}$. In $G\cong Q_8 \times \Z_2^n$, every element either inverts or commutes with every other element, so $h^{-1}x^{-1}=x^{-1}h^{\pm 1}=\varphi(x)h^{\pm 1}$. Hence the edge from $\varphi(x)$ to $\varphi(y)$ has the same color, $c(h)$, as the edge from $x$ to $y$. This shows that $\Gamma$ is not a CCA graph.

For the converse, let $\varphi$ be an arbitrary color-preserving automorphism of $\Gamma$ that fixes the vertex $1_G$. We will show that either $G$ is a Hamiltonian $2$-group, or $\varphi$ is a group automorphism of $G$.

Suppose initially that for every $g \in G$, $\varphi(g)=g^{-1}$. If $G$ is abelian then $\varphi$ is an automorphism of $G$, so we suppose that there exist $g, h$ such that $gh\neq hg$. The fact that $\varphi$ is color-preserving forces $(gh)^{-1}=\varphi(gh)\in \{\varphi(g)h,\varphi(g)h^{-1}\}=\{g^{-1}h,g^{-1}h^{-1}\}$. Since $g$ and $h$ do not commute, we see that $(gh)^{-1}\neq g^{-1}h^{-1}$ (so $(gh)^{-1}=g^{-1}h$) and that $|h|>2$. Similarly, reversing the roles of $g$ and $h$, we conclude $(hg)^{-1}=h^{-1}g$ and $|g|>2$. Thus, $h$ and $g$ invert each other. Furthermore, combining these yields $gh=g^{-1}h^{-1}$, so $g^2=h^{-2}$. But we also have $(gh)^2=ghgh=g^2=h^2$, so $h^2=h^{-2}$, meaning $|h|=|g|=4$.
Observe that since every pair of non-commuting elements invert each other, every subgroup of $G$ is normal, so that (since $G$ is nonabelian) $G$ is Hamiltonian. We have also shown that every element not in the centre of $G$ has order $4$; if $z$ is in the centre of $G$, then $gz$ does not commute with $h$, so $gz$ has order 4. Furthermore, since every pair of non-commuting elements have equal squares, $(gz)^2=g^2z^2=h^2$ and since $g^2=h^2$, we see that $|z|=2$. Thus, $G$ is a Hamiltonian $2$-group.

We may now assume that there is some $g \in G \setminus 1_G$ such that $\varphi(g)=g$ and $|g|>2$. If $\varphi=1$ then we are done, so there must be some $x \in G$ such that $\varphi(x)=x^{-1}\neq x$.

Suppose that the elements fixed by $\varphi$ form a nontrivial proper subgroup $H$ of $G$; that is, whenever $g_1, g_2\in G$ are fixed by $\varphi$, then so is $g_1g_2$. Since $\varphi(x)=x^{-1}$, by \cref{inverters}, $x$ inverts $g_1, g_2$, and $g_1g_2$, so $x^{-1}g_1g_2x=x^{-1}g_1xx^{-1}g_2x=g_1^{-1}g_2^{-1}=(g_1g_2)^{-1}$. This implies that $g_1$ and $g_2$ commute. Thus, $H$ is abelian. Since $|x|=4$ (by \cref{inverters}), $x^2$ is fixed by $\varphi$, so $x^2 \in H$. If $H$ has index 2 in $G$ then by the definition of $H$, every $xh \in xH$ is inverted by $\varphi$. It is straightforward to show that $\varphi$ is a group automorphism of $G$ in this case (in fact, such a $G$ is a generalised dicyclic group, and this map $\varphi$ is a well-known automorphism of such groups).  If the index of $H$ is greater than 2 then there is some $y\not\in H \cup xH$, so $y$ is inverted by $\varphi$. Furthermore, since $y \not\in xH=x^{-1}H$, we see that $xy \not\in H$, so $xy$ is also inverted by $\varphi$. But then $x, y$, and $xy$ all invert every $g \in H$, which is not possible (if $x$ and $y$ invert $g$ then $xy$ commutes with $g$).

We may now assume that the elements fixed by $\varphi$ do not form a subgroup of $G$, so there exist $a,b \in G$ with $\varphi(a)=a$, $\varphi(b)=b$, and $\varphi(ab)=(ab)^{-1}$. By \cref{inverters}, $ab$ inverts $a$ and $b$, so $a$ and $b$ invert each other, and $|ab|=4$. Hence $a^2=b^2=(ab)^2$. This is enough to characterise $\langle a,b \rangle \cong Q_8$.
Relabel with the standard notation for $Q_8$, so $i=a$ and $j=b$. We will show that for every element $y$ of $G$, $y^2=\pm 1$.

First, let $h$ be an arbitrary element of $G$ that is not in $\langle i,j \rangle$, and that is inverted by $\varphi$. By \cref{inverters}, $|h|=4$ and $h$ inverts $i$ and $j$. Also, since $h$ inverts $i$ and $j$ and $k=ij$, we see that $h$ must commute with $k$. Observe that if $hk$ were inverted by $\varphi$, then $hk$ would invert $i$ (and $j$), but this is impossible since $h$ and $k$ each invert $i$. So $hk$ must be fixed by $\varphi$. Thus $h$ and $k$ both invert $hk$, which implies that $h^2=k^2=-1$, as desired.

Let $g$ be an arbitrary element of $G$ that is not in $\langle i,j\rangle$, and that is fixed by $\varphi$. If $|g|=2$ then $g^2=1$, as claimed. So we assume $|g|>2$. Suppose that $gi$ is fixed by $\varphi$. Then $k$ inverts $gi$, so $k^{-1}gik=i^{-1}g^{-1}=g^{-1}i^{-1}$ since $k$ inverts both $g$ and $i$, so $i$ and $g$ commute. The same is true for $gj$; thus, if $gi$ and $gj$ were both fixed by $\varphi$, then $g$ would commute with $i$ and $j$ and hence with $k$, a contradiction since $|g|>2$. So at least one of $gi$ and $gj$ is inverted by $\varphi$. Suppose that $gi$ is inverted by $\varphi$. By the argument of the previous paragraph, $(gi)^2=-1$. Also, $gi$ inverts $i$, so $g$ inverts $i$. Hence $-1=(gi)^2=g(gi^{-1})i=g^2$, completing the proof of our claim.

We have shown that every non-identity element of $G$ has order $2$ or $4$ (so $G$ is a 2-group), and that the elements of order $4$ all square to $-1$. To complete the proof that $G$ is Hamiltonian, we need to show that every subgroup is normal. Let $r, s \in G$. Then $(rs)^2=\pm 1$, so $rsr=\pm s$, so $r^{-1}sr = rsr$ if $r$ has order $2$ or $r^{-1}sr=-rsr$ otherwise.  In any case, $r^{-1}sr = \pm s$. Similarly, $s^{-1}rs=\pm r$. Thus every subgroup is normal in $G$, so $G$ is Hamiltonian.
\end{proof}

For the reader's convenience, we state the parts we will be using of what is proved in \cite[Theorem 3.2]{Dobson2006a}.

\begin{thrm}[{\cite[Theorem 3.2]{Dobson2006a}}]\label{Teds}
Let $G$ be a $2$-closed group of degree $pq$, $p>q$, admitting a nontrivial invariant partition $\mathcal B$, and suppose that $G$ does not contain a regular cyclic subgroup. Then
\begin{enumerate}
\item if $G$ is the automorphism group of a graph, and every nontrivial invariant partition $\mathcal B$ admitted by $G$ consists of $p$ blocks of size $q$ then $p=2^{2^s}+1$ is a Fermat prime (and $q$ divides $p-2$).
\item if $\mathcal B$ is formed by the orbits of a normal subgroup then $\mathcal B$ consists of $q$ blocks of size $p$, and $G$ contains a transitive metacyclic subgroup. If a maximal transitive metacyclic subgroup $H$ has order $ap$, then
\begin{enumerate}
\item if $\fix_G(\mathcal B)\vert_B$ is doubly transitive for $B \in \mathcal B$, then $q=2$ and
\begin{enumerate}
\item $G=\Z_2 \ltimes \PSL(2,11)$, or
\item $G=\Z_2 \ltimes \PGammaL(n,k)$, where $n$ is prime, $k=r^m$, $r$ is prime, $\gcd(n,k-1)=1$, and $m$ is a power of $n$; while
\end{enumerate}
\item if $\fix_G(\mathcal B)\vert_B$ is not doubly transitive, then
\begin{enumerate}
\item if $a\neq q$ then $G$ is metacyclic of order $ap$, and
\item if $a=q$ then $H\triangleleft G$.
\end{enumerate}
\end{enumerate}
\end{enumerate}
\end{thrm}

It should be noted that in the proof of part (1) of \cite[Theorem 3.2]{Dobson2006a}, it is deduced that $\fix_G(\mathcal B)=1$, so that $\mathcal B$ is not formed by the orbits of a normal subgroup, so the statement above reflects the fact that only part (2) can arise if $\mathcal B$ is formed by the orbits of a normal subgroup. Furthermore, what we have stated for part (2bii) requires combining the statements of Lemma 2.5 and Theorem 3.2 (2bii) of \cite{Dobson2006a}.

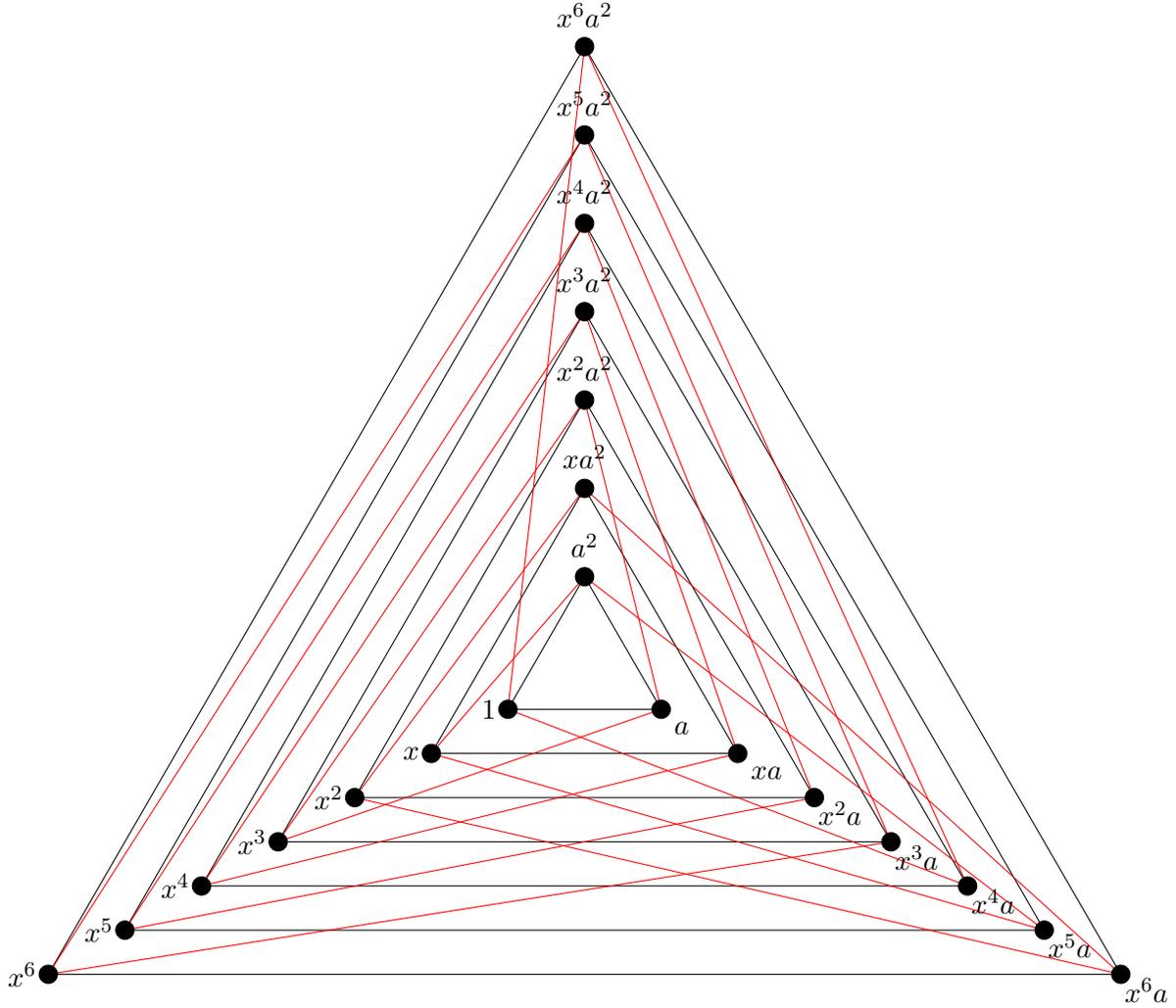
\begin{figure}
\begin{tikzpicture}
\def \dif {1.2}
\foreach \t in {1,...,3}
{
 \foreach \r in {1,...,7}
 {
 \draw({90+ \t*120}:{\dif*\r})--({90+ (\t-1)*120}:{\dif*\r});
 }
}
\foreach \r in {0,...,6}
{
\pgfmathsetmacro{\rtwo}{mod(\r+1,7)}
\pgfmathsetmacro{\rthree}{mod(\r+5,7)}
 \draw[color=red](90:{\dif*\r+\dif})--(210:\dif*\rtwo+\dif)--(330:{\dif*\rthree+\dif})--cycle;
}

\foreach \r in {3,...,7}
 {
 \node[label=above:$x^{\pgfmathparse{\r-1}\pgfmathprintnumber{\pgfmathresult}}a^2$] () at (90:{\dif*\r}) [circle, draw, fill=black!100, inner sep=2.5pt, minimum width=5pt]{};
 }
\node[label=above:$xa^2$] () at (90:{2*\dif}) [circle, draw, fill=black!100, inner sep=2.5pt, minimum width=5pt]{};
\node[label=above:$a^2$] () at (90:{\dif}) [circle, draw, fill=black!100, inner sep=2.5pt, minimum width=5pt]{};
\foreach \r in {3,...,7}
 {
 \node[label={[label distance=-3pt]left:$x^{\pgfmathparse{\r-1}\pgfmathprintnumber{\pgfmathresult}}$}] () at (210:{\dif*\r}) [circle, draw, fill=black!100, inner sep=2.5pt, minimum width=5pt]{};
 }
\node[label={[label distance=-3pt]left:$x$}] () at (210:{2*\dif}) [circle, draw, fill=black!100, inner sep=2.5pt, minimum width=5pt]{};
\node[label={[label distance=-3pt]left:$1$}] () at (210:{\dif}) [circle, draw, fill=black!100, inner sep=2.5pt, minimum width=5pt]{};
\foreach \r in {3,...,7}
 {
 \node[label={[label distance=-7pt]below right:$x^{\pgfmathparse{\r-1}\pgfmathprintnumber{\pgfmathresult}}a$}] () at (330:{\dif*\r}) [circle, draw, fill=black!100, inner sep=2.5pt, minimum width=5pt]{};
 }
\node[label={[label distance=-2pt]below right:$xa$}] () at (330:{2*\dif}) [circle, draw, fill=black!100, inner sep=2.5pt, minimum width=5pt]{};
\node[label={[label distance=-2pt]below right:$a$}] () at (330:{\dif}) [circle, draw, fill=black!100, inner sep=2.5pt, minimum width=5pt]{};
\end{tikzpicture}
\caption{The unique non-CCA Cayley graph of $F_{21}$.}\label{F21nCCA}
\end{figure}

\begin{prop}\label{uniqueF21nonCCA}
The Cayley graph $\Cay(F_{21},\{a^{\pm 1},(ax)^{\pm 1}\})$ is the only non-CCA graph of
$F_{21}$ up to isomorphism.  It is regular of valency $4$, is an orbital graph of the
primitive group $\PGL(2,7)$, and has imprimitive color-preserving automorphism group
${\mathcal A}^o = \PSL(2,7)$.
\end{prop}

\begin{proof}
Let $\Gamma = \Cay(F_{21},S)$ be a non-CCA graph of $F_{21}$.  Then $(F_{21})_L\le {\mathcal A}^o$ is not normal in ${\mathcal A}^o$, which is $2$-closed by Lemma \ref{color preserving 2-closed}.  As $7$ is not a Fermat prime, by \cref{Teds}(1) $\Aut(\Gamma)$ is either primitive, contains a regular cyclic subgroup, or has an invariant partition ${\mathcal B}$ with blocks of size $7$ formed by the orbits of a normal subgroup.  Observe that in the latter case if $G$ is any transitive subgroup of $\Aut(\Gamma)$, then $G$ contains at least one element of order $7$. Since $G/{\mathcal B}\le S_3$ has order coprime to $7$, every element of order $7$ must lie in $\fix_G({\mathcal B})$, so (using transitivity) the orbits of $\fix_G({\mathcal B})$ are the blocks of ${\mathcal B}$.

Suppose that ${\mathcal A}^o$ does not contain a regular cyclic subgroup and $\Aut(\Gamma)$ is not primitive.  Then ${\mathcal A}^o$ has an invariant partition $\mathcal B$ with blocks of size $7$ formed by the orbits of a normal subgroup.  As $21$ is odd, it cannot be the case that \cref{Teds}(2a) occurs.  But either part of \cref{Teds}(2b) implies that $(F_{21})_L\tl\Aut(\Gamma)$, a contradiction.  Thus either ${\mathcal A}^o$ contains a regular cyclic subgroup or $\Aut(\Gamma)$ is primitive.

Suppose that ${\mathcal A}^o$ contains a regular cyclic subgroup and $\Aut(\Gamma)$ is imprimitive, admitting $\mathcal B$.  The $2$-closed permutation groups of degree a product of two distinct primes $p$ and $q$ that contain a regular cyclic subgroup and are imprimitive with blocks of size $p$ are known \cite{KlinP1981} (another proof is given in \cite{Dobson2006a}, while descriptions for square-free $n$ are given in \cite{DobsonM2005} and all integers in \cite{Li2005}).  They are either contained in $N_{S_{pq}}((\Z_{pq})_L)$, or are permutation equivalent to $A\wr B$ or $A\times B$ where $A$ and $B$ are $2$-closed groups of degree $q$ and $p$, respectively. Furthermore, in the last of these cases, at least one of $A$ and $B$ must be a symmetric group on at least 5 points or we have $A\times B\le N_{S_{pq}}((\Z_{pq})_L)$; in our case, this means $\mathcal A^o$ is equivalent to $A \times S_7$.
As by \cite[Lemma 6.2]{HujdurovicKMM2016} we have $\Stab_{{\mathcal A}^o}(x)$ is a $2$-group for every vertex $x$ we see that the only possibility is
${\mathcal A}^o\le N_{S_{21}}((\Z_{21})_L)$, so $|\mathcal A^o| \mid 2^2 \cdot 3^2 \cdot 7=|N_{S_{21}}((\Z_{21})_L)|$.  We see that $\mathcal A^o$ has a unique Sylow $7$-subgroup of order 7 that lies in both $F_{21}$ and $(\Z_{21})_L$, and whose orbits are the blocks of $\mathcal B$. However, the elements of order $3$ in these two regular groups (each of which acts cyclically of order $3$ on the blocks of $\mathcal B$) are not the same, since one centralizes the Sylow $7$-subgroup but the other does not. Therefore, by combining one such element from $F_{21}$ with one such element of $(\Z_{21})_L$, we can obtain an element of $\fix_{\mathcal A^o}(\mathcal B)$ that is not the identity, but that fixes some point. Calculations show that this element in fact has order $3$, which again contradicts \cite[Lemma 6.2]{HujdurovicKMM2016}.

Finally, suppose that $\Aut(\Gamma)$ is primitive. We know that $\Aut(\Gamma)$ is not doubly-transitive because this would imply $\Gamma$ being complete (since it is connected, it cannot be empty), and \cref{complete-CCA} shows that the complete Cayley graph on $F_{21}$ is CCA. There are only two primitive permutation groups of degree 21: $\PGL(2,7)$, and $A_7$ in its action on ordered pairs \cite[Appendix B]{DixonM1996}. But $A_7$ in this action does not contain a regular subgroup, by \cite[Lemma 3.1]{PraegerX1993}. So we must have $\Aut(\Gamma)\cong \PGL(2,7)$. We know that  $\mathcal A^o$ contains $F_{21}\le \PSL(2,7)\triangleleft \PGL(2,7)$. Since $F_{21}$ is a maximal subgroup of $\PSL(2,7)$ by the ATLAS of Finite Group Representations \cite{Atlas}, we have $\mathcal A^o \cap \PSL(2,7)$ is either $F_{21}$ or $\PSL(2,7)$.

Suppose $\mathcal A^o\cap \PSL(2,7)=(F_{21})_L$. Then $\mathcal A^o$ is one of the maximal subgroups of $\PGL(2,7)$ that do not contain $\PSL(2,7)$; but these are $F_{21}$ and $F_{42}= \Z_7 \rtimes Z_6$, and $F_{21}$ is normal in each of these, so such a graph would be CCA. So we must have $\mathcal A^o \cap \PSL(2,7)=\PSL(2,7)$. Since $\PSL(2,7)$ is maximal (of index 2 in fact) in $\PGL(2,7)$, and $\PGL(2,7)$ is primitive but $\mathcal A^o$ is imprimitive by \cref{calAoimprimitive}, we must have $\mathcal A^o\cong\PSL(2,7)$.

Now, by \cite[Example 2.3]{HujdurovicKMM2016} $\Gamma_1 = \Cay(F_{21},\{a^{\pm 1},(ax)^{\pm 1}\})$ is a non-CCA graph of $F_{21}$ and of course has valency $4$.  By the above argument, we must have $\Aut(\Gamma_1)=\PGL(2,7)$. As $\PGL(2,7)$ has three suborbits of lengths $4,8,8$ by the ATLAS of Finite Group Representations \cite{Atlas}, we conclude that $\Gamma_1$ is an orbital graph of $\PGL(2,7)$.  One can then check, for example with MAGMA \cite{MAGMA}, that $\Gamma_1$ is the only Cayley graph of $F_{21}$ with automorphism group $\PGL(2,7)$ that is not a CCA graph of $F_{21}$.
\end{proof}

\begin{hey}
{\rm While the Cayley graph $\Gamma = \Cay({\mathbb F}_{21},\{a^{\pm 1},(ax)^{\pm 1}\})$ given in preceding Proposition is unique up to isomorphism, there are $21$ different choices for $S$ which will yield a graph isomorphic to $\Gamma$, each of which is the image of $\Gamma$ under an automorphism of ${\mathbb F}_{21}$ by an element of the unique subgroup of $\Aut({\mathbb F}_{21})$ of order $21$. To see this, recall that $\Aut(\Gamma)\cong \PGL(2,7)$ is primitive. Now, $\Z_{21}$ is a Burnside group, so that a primitive group containing a regular copy of $\Z_{21}$ must be doubly-transitive, which $\Aut(\Gamma)$ clearly is not. Thus, $\Aut(\Gamma)$ does not contain a regular cyclic subgroup. Furthermore, $7^2$ does not divide $|\Aut(\Gamma)|=336$. Thus, by
 \cite[Theorem 9]{Dobson1998} any Cayley graph of ${\mathbb F}_{21}$ is isomorphic to $\Gamma$ if and only if an isomorphism between the two graphs is in $\Aut({\mathbb F}_{21})$.  It is not hard to show this group has order $42$ (it is isomorphic to $F_{42}=\Z_7 \rtimes \Z_6$, and consists of the inner automorphisms together with the outer automorphism of order $2$), and that $\Aut({\mathbb F}_{21})\cap\Aut(\Gamma)$ has order $2$ (only the outer group automorphism acts as a graph automorphism).}
\end{hey}

\section{Structure of non-CCA graphs of odd square-free order}

\begin{defin}\label{Cartesianproduct}
The {\bf Cartesian product} $G\boxprod H$ of two graphs $G$ and $H$ is the graph with vertex set $V(G)\times V(H)$ and edge set $\{(u,v)(u',v'):u = u'{\rm\ and \ } u'v'\in E(H){\rm\ or\ }v = v'{\rm\ and\ }uu'\in E(G)\}$.
\end{defin}

\begin{defin}
Let $G$ be a group and $K\le G$.  We denote the subgroup $\{k_L:k\in K\}$ of $G_L$ by $\hat{K}_L$.
\end{defin}

We now give a sufficient condition for a Cayley color graph (with the natural edge-coloring defined in \cref{def:coloring}) of odd order to be a Cartesian product. This will be a crucial tool in our main result.

\begin{lem}\label{maintool}
For a connected Cayley color graph (with the natural edge-coloring defined in \cref{def:coloring}) $\Cay(G,S)$ of a group $G$ of odd order, let ${\mathcal B}$ be an ${\mathcal A}$-invariant partition where $G_L\le {\mathcal A}\le{\mathcal A}^o$ is transitive.  Define an equivalence relation $\equiv$ on $G$ by $g\equiv h$ if and only if $\Stab_{\fix_{{\mathcal A}}({\mathcal B})}(g) = \Stab_{\fix_{{\mathcal A}}({\mathcal B})}(h)$. Let $\Gamma$ be the graph obtained from $\Cay(G,S)$ by removing all edges both of whose endpoints are contained in some block $B\in{\mathcal B}$. The following hold:
\begin{enumerate}
\item the equivalence classes of $\equiv$ form an ${\mathcal A}$-invariant partition ${\mathcal E}$,
\item\label{fixercon2} each connected component of $\Gamma$ is contained in some $E\in{\mathcal E}$,
\item\label{fixercon3} if each $E\in{\mathcal E}$ contains exactly one element from each $B\in{\mathcal B}$, then $\Cay(G,S) \cong \Gamma_1\boxprod\Gamma_2$, where $\Gamma_1$ is a connected component of $\Gamma$ and
$\Gamma_2 = \Cay(G,S)[B_1]$, with $B_1\in{\mathcal B}$. Furthermore, there exist $G_1,G_2 \le G$ such that $\Gamma_1\cong
\Cay(G_1,S_1)$, and $\Gamma_2=\Cay(G_2,S_2)$, where $S_i=S\cap G_i$ for $i=1,2$ and
$S=S_1\cup S_2$, and the natural edge-coloring given in \cref{def:coloring} is preserved under the isomorphism; and $G=G_1\times G_2$.
\end{enumerate}
\end{lem}

\begin{proof}
As $G_L\le {\mathcal A}$, any invariant partition of ${\mathcal A}$ is also a $G_L$-invariant partition.  The partition ${\mathcal E}$ is an ${\mathcal A}$-invariant partition as the equivalence relation $\equiv$ is an ${\mathcal A}$-congruence \cite[Exercise 1.5.4]{DixonM1996}.  By \cite[Theorem 1.5A]{DixonM1996}, there exists a subgroup $G_2\le G$ such that ${\mathcal B}$ consists of the orbit of $(G_2)_L = \{k_L:k\in G_2\}$ that contains $1_G$ together with its images under $G_L$.  Now let $B_g\in {\mathcal B}$ with $g\in B_g$, and $h = gs\in G$ such that $h\not\in B_g$ and $s\in S$ (so $(g,gs)\in E(\Gamma)$).  Then $h\in B_h\in{\mathcal B}$ for some $B_h\not = B_g$.

Let $\alpha\in \Stab_{\fix_{{\mathcal A}}({\mathcal B})}(g)$.  We first claim that $\alpha(h) = h$.  Indeed, as $\alpha\in{\mathcal A}^o$ and $\alpha(g) = g$, we have that $\alpha(h) = h = gs$ or $\alpha(h) = gs^{-1}$.  If $\alpha(h) = gs^{-1}$, then as $\alpha(B) = B$ for all $B\in{\mathcal B}$, it must be the case that $gs^{-1},gs\in B_h$.  
So there exists $k\in G_2$ such that $gsk = gs^{-1}$ so that $s^2 = k^{-1}$.  As $k^{-1}\in  G_2$ while $s\not\in G_2$ we have that $\la k\ra < \la s\ra$ and as $s^2\in\la k\ra$ we see that the order of $s$ is even.  However, $G$ has odd order, a contradiction.  Thus $\alpha(h) = h$ completing the claim.

The claim implies that if $\alpha\in\fix_{{\mathcal A}}({\mathcal B})$ fixes $1_G$, then it fixes the neighbors of $1_G$ not contained in $B\in{\mathcal B}$ with $1_G\in B$.  Arguing inductively, $\alpha$ fixes all vertices that are words formed by elements contained in $S_1=S\backslash B$.  We conclude that $\alpha$ fixes every vertex in the connected component of $\Gamma$ that contains $1_G$. Clearly all of these vertices are contained in the equivalence class of $\equiv$ that contains $1_G$, and (\ref{fixercon2}) follows. This also shows that the vertices of this connected component consist of the subgroup $G_1$ of $G$ that is generated by $S_1$.

Now we additionally assume that $\Stab_{\fix_{{\mathcal A}}({\mathcal B})}(1_G)$ fixes exactly one element from every block of ${\mathcal B}$.  We have seen that $\Gamma_1=\Cay(G_1,S_1)$ where $S_1=S\setminus B$, and $\Gamma_2=\Cay(G_2,S_2)$ where $S_2=S\cap B$, so $S=S_1\cup S_2$. For each $g\in G$, let $B_g\in{\mathcal B}$ contain $g$ and $E_g\in{\mathcal E}$ contain $g$. As $\Cay(G,S)$ is connected and $G$ is transitive, it must be the case that $\Gamma_1$ is isomorphic to the induced subgraph on $E_g$, and that $\Gamma_2$ is connected.  Also, an element $g\in G$ is uniquely determined by the pair $(E_g,B_g)\in{\mathcal E}\times {\mathcal B}$ where $g\in E_g$ and $g\in B_g$.  Define $\delta:G\rightarrow {\mathcal E}\times{\mathcal B}$ by $\delta(g) = (E_g,B_g)$.  As $V(\Gamma_1) = E_{i}$ for some $i$ and $\vert E_i\cap B\vert = 1$ for every $B\in{\mathcal B}$, we identify each vertex of $V(\Gamma_1)$ uniquely with the block of ${\mathcal B}$ it is contained in.  Similarly, as $V(\Gamma_2) = B_1$, we identify each vertex of $V(\Gamma_2)$ with the block of ${\mathcal E}$ that it is contained in.  We claim that $\delta$ is an isomorphism between $\Cay(G,S)$ and $\Gamma_1\boxprod\Gamma_2$ that preserves the edge colors.

Let $e=gh\in E(\Cay(G,S))$.  As every edge of $\Cay(G,S)$ is either contained in $\Gamma$ (recall that $\Gamma$ is the graph obtained from $\Cay(G,S)$ by removing all edges both of whose endpoints are contained in some block $B\in{\mathcal B}$) or $\Cay(G,S)\backslash E(\Gamma)$, both endpoints of $e$ are contained in a component of $\Gamma$ (which is the same as a block of ${\mathcal E}$) or a block of ${\mathcal B}$.  In the former case, $\delta(gh) = (E_g,B_g)(E_h,B_h)$ and $E_g = E_h$ while $B_gB_h\in E(\Gamma_1)$, so that $h=gs$ for some $s \in S_1$ and this color is preserved by the isomorphism.  In the latter case, $\delta(gh) = (E_g,B_g)(E_h,B_h)$ and $B_g = B_h$ while $E_gE_h\in E(\Gamma_2)$, so that $h=gs$ for some $s \in S_2$ and this color is preserved by the isomorphism.  So $\delta$ is an isomorphism as claimed. We still need to show that $G=G_1\times G_2$.

It is clear that $G_1\cap G_2=\{1\}$ and that $|G|=|G_1||G_2|$, so we need only show that elements of $G_1$ commute with elements of  $G_2$. Observe that for any $s_1 \in S_1$ and any $s_2 \in S_2$,  $(E_{s_1s_2},B_{s_1s_2})=(E_{s_1s_2},B_{s_1})=(E_{s_2},B_{s_1})$. Similarly, $(E_{s_2s_1},B_{s_2s_1})=(E_{s_2},B_{s_2s_1})=(E_{s_2},B_{s_1})$. Thus, $s_1s_2=s_2s_1$, and since $G_1=\langle S_1 \rangle$ and $G_2=\langle S_2\rangle$,  we have $G=G_1\times G_2$ and (\ref{fixercon3}) follows.
\end{proof}


We need three more preliminary results before turning to our main result.

\begin{lem}\label{quotient color-preserving}
Let $G$ be a group and $S\subseteq G$.  Let $N\triangleleft G$, ${\mathcal B}$ be the orbits of $N_L$ and ${\mathcal A}^o$ the group of color-preserving automorphism of $\Cay(G,S)$.  If ${\mathcal B}$ is an ${\mathcal A}^o$-invariant partition, and $\alpha\in{\mathcal A}^0$, then $\alpha/{\mathcal B}$ is also a color-preserving automorphism of $\Cay(G/N,S/N)$, where $S/N = \{sN:s\in S\}$.
\end{lem}

\begin{proof}
Let $\alpha\in\Aut(\Cay(G,S))$ be a color-preserving automorphism.   As ${\mathcal A}^o/{\mathcal B}\le\Aut(\Cay(G,S)/{\mathcal B})$, we have $\alpha/{\mathcal B}\in \Aut(\Cay(G,S)/{\mathcal B})$.  Let $sN\in S/N$ and $g\in G$.  As $\alpha$ is a color-preserving automorphism of $\Cay(G,S)$, we have $\alpha(g,gs) = (h,hs)$ or $(h,hs^{-1})$ for some $h\in G$.  Then $\alpha(gN,(gN)(sN))=\alpha(gN,gsN) = (hN,hsN)$ or $(hN,hs^{-1}N)$. Since $hsN=(hN)(sN)$ and $hs^{-1}N=(hN)(s^{-1}N)=(hN)(sN)^{-1}$, we see that $\alpha/{\mathcal B}$ is a color-preserving automorphism of $\Cay(G/N,S/N)$.
\end{proof}

\begin{lem}\label{bottomsemiregular}
Let $G$ be a group of odd square-free order, and $\Cay(G,S)$ a Cayley graph such that ${\mathcal A}^o$ admits a normal invariant partition ${\mathcal B}$. Then the orbits of $\fix_{G_L}(\mathcal B)$ are the blocks of $\mathcal B$. In particular, if $\fix_{\mathcal A^o}(\mathcal B)$ is semiregular, then $\fix_{\mathcal A^o}(\mathcal B)=\fix_{G_L}(\mathcal B)$.
\end{lem}

\begin{proof}
Let $k$ be the size of each $B \in \mathcal B$, so $k$ is odd and square-free. Suppose to the contrary that for some prime divisor $p$ of $k$, $p$ does not divide $|\fix_{G_L}(\mathcal B)|$. Now, a Sylow $p$-subgroup $P$ of $\fix_{{\mathcal A}^o}({\mathcal B})$ has order $p^i$ for some $i \ge 1$. Let $P'$ be a Sylow $p$-subgroup of $G_L$. Then $P'$ has order $p$, and by assumption does not lie in $\fix_{{\mathcal A}^o}({\mathcal B})$. Since $P$ lies in a normal subgroup (namely, $\fix_{{\mathcal A}^o}({\mathcal B})$) of $\mathcal A^o$ that does not contain $P'$, we see that $P$ and $P'$ are not conjugate in $\mathcal A^o$. Thus, $P$ and $P'$ cannot both be Sylow $p$-subgroups of $\mathcal A^o$. This means that $p^2$ must divide $|\mathcal A^o|$. Hence $\Stab_{{\mathcal A}^o}(1)$ is divisible by $p$.  This implies $p = 2$ (since every point-stabilizer of $\mathcal A^o$ is a $2$-group by \cite[Lemma 6.3]{HujdurovicKMM2016}), a contradiction. We conclude that $k$ divides $|\fix_{G_L}(\mathcal B)|$, so that the orbits of $\fix_{G_L}(\mathcal B)$ are the blocks of $\mathcal B$.

If $\fix_{\mathcal A^o}(\mathcal B)$ is semiregular, then $|\fix_{\mathcal A^o}(\mathcal B)| \le k$. Since $\fix_{G_L}(\mathcal B) \le \fix_{\mathcal A^o}(\mathcal B)$ is semiregular of order $k$, we must have $\fix_{\mathcal A^o}(\mathcal B)=\fix_{G_L}(\mathcal B)$.
\end{proof}

\begin{lem}\label{cartesian-CCA}
Let $G$ be a group, and $\Gamma=\Cay(G,S)$ a Cayley graph on $G$. Suppose that $G=G_1 \times G_2$, where $\Gamma_1=\Cay(G_1,S_1)$ and $\Gamma_2=\Cay(G_2,S_2)$, for some $S_1\subset G_1,S_2 \subset G_2$, and that $\Gamma=\Gamma_1\boxprod\Gamma_2$. Further assume that $\Gamma_1$ and $\Gamma_2$ have no common factors with respect to Cartesian decomposition.
If $\Gamma_1$ and $\Gamma_2$ are both CCA graphs (on $G_1$ and $G_2$ respectively), then $\Gamma$ is a CCA graph on $G$.
\end{lem}

\begin{proof}
By \cite[Corollary 15.6]{ImrichKR2008}, we have $\Aut(\Gamma) = \Aut(\Gamma_1)\times\Aut(\Gamma_2)$. Although their statement assumes that we have a prime factorization (with respect to the Cartesian product), this only affects the outcome if some prime factor is common to $\Gamma_1$ and $\Gamma_2$, which we have assumed is not the case.

If $\alpha$ is an automorphism of $\Gamma$ that fixes $1_G$, then this implies that $\alpha$ can be written as $(\alpha_1, \alpha_2)$, where $\alpha_i\in \Aut(\Gamma_i)$ for $i=1,2$. Since $\Gamma_i$ is a CCA graph on $G_i$, this implies that $\alpha_i$ is an automorphism of $G_i$. It is then easy to see that $\alpha=(\alpha_1, \alpha_2)$ is an automorphism of $G=G_1\times G_2$.
\end{proof}

The following result is the main result of this paper.

\begin{thrm}
Let $n$ be odd and square-free, $G$ a group of order $n$, and $\Gamma = \Cay(G,S)$ be a connected Cayley graph of $G$.  Then $\Gamma$ is a non-CCA Cayley graph of $G$ if and only if
\begin{enumerate}
\item\label{con1} $n$ is divisible by $21$ and $G = G_1\times F_{21}$, where $G_1$ is a group of order $n/21$ and $F_{21}$ is the nonabelian group of order $21$, and
\item\label{con2} $\Gamma = \Gamma_1\boxprod\Gamma_2$, where $\Gamma_1$ is a CCA graph of order $n/21$ and $\Gamma_2$ is the unique non-CCA graph of order $21$.
\end{enumerate}
\end{thrm}

\begin{proof}
Suppose that $G$ and $\Gamma$ satisfy (\ref{con1}) and (\ref{con2}).  Then $\Gamma$ is not CCA by \cite[Proposition 3.1]{HujdurovicKMM2016}.

Conversely, we proceed by induction on $n$, with the base case being $n = 21$ as by \cite[Theorem 6.8]{HujdurovicKMM2016} that is the smallest positive odd square-free integer for which there is a non-CCA Cayley graph of some group $G$, and that group is $G = F_{21}$.  By Proposition \ref{uniqueF21nonCCA}, there is a unique non-CCA Cayley digraph of $F_{21}$.  Let $n > 21$ and assume that the result holds for all $21 \le m < n$.  Let $G$ be a group of order $n$, and $\Gamma = \Cay(G,S)$ be connected.

Let $N$ be a minimal normal subgroup of ${\mathcal A}^o$, so that either $N$ is an elementary abelian $p$-group for some prime $p\vert n$, or $N$ is a direct product of isomorphic simple groups $T$.  The orbits of $N$ form an invariant partition ${\mathcal B}$ of ${\mathcal A}^o$.

Suppose ${\mathcal B} = \{G\}$ is a trivial invariant partition.  If $N$ is an elementary abelian $p$-group, then as $n$ is square-free we have $n = p$ is prime.  Now by Burnside's Theorem \cite[Theorem 3.5B]{DixonM1996}, ${\mathcal A}^o\le\AGL(1,p)$ so every element of ${\mathcal A}^o$ is affine, contradicting our assumption that $\Gamma$ is not CCA.  Otherwise, as $n$ is square-free, the O'Nan Scott Theorem implies that $\soc({\mathcal A}^o)$ is a simple group (see for example \cite[Lemma 2.1]{DobsonMMN2007}).  Additionally, by \cite[Lemma 6.3]{HujdurovicKMM2016} $\Stab_{\soc({\mathcal A}^o)}(v)$ is a $2$-group for any vertex $v$, and as observed in the proof of \cite[Theorem 6.8]{HujdurovicKMM2016}, the only nonabelian simple group that has a $2$-complement is $\PSL(2,p)$ where $p$ is a Mersenne prime.   Then $\Stab_{\soc({\mathcal A}^o)}(v) = D_{p + 1}$, which is a Sylow $2$-subgroup of $\PSL(2,p)$ as $p$ is a Mersenne prime. This then implies that $n = p(p-1)/2$.  Additionally, by Theorem \ref{bigpreviousresult} $G$ has a section isomorphic to $F_{21}$.  As $\soc({\mathcal A}^o) = \PSL(2,p)$, we see ${\mathcal A}^o\le \PGL(2,p)$, and $[\PGL(2,p):\PSL(2,p)] = 2$.  We conclude that every element of odd order in ${\mathcal A}^o$ is contained in $\PSL(2,p)$.  In particular, $G\le\PSL(2,p)$.  Consulting Dickson's Classification of the subgroups of $\PSL(2,q)$ \cite[Hauptsatz 8.27]{Huppert1967}, we see $G$ is a semi-direct product of an elementary abelian $p$-group of order $p$ and a cyclic group of order $t$ where $t$ is a divisor of $p - 1$.  So the commutator subgroup of $G$ is contained in a subgroup of $G$ of order $p$.  As $F_{21}$ is a section of $G$ and so the commutator subgroup of $G$ contains an element of order $7$, we conclude $p = 7$ so $n=21$, contradicting $n>21$.

Suppose ${\mathcal B}\not = \{G\}$, then remembering that $n$ is odd and applying \cite[Theorem 2.10]{DobsonMMN2007}, it follows that $N$ in its action on $B\in{\mathcal B}$, written $N\vert_B$, is either a simple group $T$ or $A_7^2$ of degree $105$.  As $\Stab_{{\mathcal A}^o}(v)$ is a $2$-group for every $v\in G$ by \cite[Lemma 6.3]{HujdurovicKMM2016} we see that $N\vert_B = T$. As noted earlier, the only nonabelian simple group that has a $2$-complement is $\PSL(2,p)$ where $p$ is a Mersenne prime.  Then $N\vert_B = T$ for every $B\in{\mathcal B}$ and $T$ is either cyclic of prime order or $T = \PSL(2,p)$ where $p$ is a Mersenne prime.

Now consider the action of $\fix_{{\mathcal A}^o}({\mathcal B})$ on $B\in{\mathcal B}$, and suppose that this action is not faithful.  Let $K$ be the kernel of this action, and $B'\in{\mathcal B}$ with $K\vert_{B'}\not = 1$.  As $K$ stabilizes each element of $B$, it must be the case that $K$ is a $2$-group by \cite[Lemma 6.3]{HujdurovicKMM2016}, and so $K\vert_{B'}$ is a normal $2$-subgroup of $\fix_{{\mathcal A}^o}({\mathcal B})\vert_{B'}$.  The orbits of $K\vert_{B'}$ then form an invariant partition of $\fix_{{\mathcal A}^o}({\mathcal B})\vert_{B'}$ with blocks of size a power of $2$, contradicting the assumption that $n$ is odd.  Thus the action of $\fix_{{\mathcal A}^o}({\mathcal B})$ on $B\in{\mathcal B}$ is faithful.

Now suppose that $\fix_{{\mathcal A}^o}({\mathcal B})$ is not semiregular of prime order.  If $T=\Z_p$ then $\fix_{{\mathcal A}^o}({\mathcal B})\cong\fix_{{\mathcal A}^o}({\mathcal B})\vert_B$, $B\in{\mathcal B}$ is isomorphic to a subgroup of $\AGL(1,p)$ that is not isomorphic to $\Z_p$, while if $T=\PSL(2,p)$ then $\fix_{{\mathcal A}^o}({\mathcal B})$ has socle $\PSL(2,p)$ and $p$ is a Mersenne prime.  Let $k$ be the size of the blocks of ${\mathcal B}$, and note that $k$ is odd.  In either case, we have that $\fix_{{\mathcal A}^o}({\mathcal B})\vert_B$ has order $k\cdot 2^\ell$ for some $\ell$, and so the stabilizer of a point in $\fix_{{\mathcal A}^o}({\mathcal B})\vert_B$ is a Sylow $2$-subgroup of $\fix_{{\mathcal A}^o}({\mathcal B})\vert_B$ for every $B\in{\mathcal B}$.  Additionally, as any two point stabilizers of a group are conjugate and Sylow $2$-subgroups are certainly conjugate, the stabilizer of a point in $\fix_{{\mathcal A}^o}({\mathcal B})$ of $b\in B$ is the stabilizer of a point in $\fix_{{\mathcal A}^o}({\mathcal B})$ of $b'\in B'$, where $B,B'\in{\mathcal B}$.  We conclude that the stabilizer of $b\in B$ in $\fix_{{\mathcal A}^o}({\mathcal B})$ fixes at least one point in every block of ${\mathcal B}$.  Additionally, the stabilizer of $b\in B$ in $\fix_{{\mathcal A}^o}({\mathcal B})$ will fix exactly one point in every block of ${\mathcal B}$ provided that the stabilizer of a point $b\in B$ of $\fix_{{\mathcal A}^o}({\mathcal B})$ in its action on $B$ fixes exactly one point.  We claim that this occurs.

If $T$ is cyclic of prime order, then the stabilizer of $b\in B$ in $\fix_{{\mathcal A}^o}({\mathcal B})$ fixes exactly one point in every block of ${\mathcal B}$ as in this case $\fix_{{\mathcal A}^o}({\mathcal B})\vert_B$ is primitive.  If $T = \PSL(2,p)$, then the stabilizer of a point of $\PSL(2,p)$ in its action on $B\in{\mathcal B}$ is $D_{p + 1}$ (a Sylow $2$-subgroup of $\PSL(2,p)$ as $p$ is a Mersenne prime).  If $p\not = 7$, then $D_{p+1}$ is maximal in $\PSL(2,p)$ by \cite[Theorem 2.2]{Giudici2016}. Since the stabilizer of a point is properly contained in the setwise stabilizer of any block containing that point, we see that in this case $\PSL(2,p)$ is primitive in its action on $B\in{\mathcal B}$.  Consequently, the stabilizer of a point in $\PSL(2,p)$ in its action on $B\in{\mathcal B}$ fixes exactly one point, and so the stabilizer of a point in the action of $\fix_{{\mathcal A}^o}({\mathcal B})$ on $B\in{\mathcal B}$ fixes exactly one point.  If $p = 7$, then $T \cong \PSL(2,7)$ which has order $168$. Since $n$ is odd and square-free, this forces each block $B$ to have length dividing $21$. The action of $T\vert_B$ can only be imprimitive if the block $B$ has length 21. In this case, $D_8$ is contained in $S_4$ and is not maximal, but the stabilizer of a point in $\PSL(2,7)$ under its action on $B\in{\mathcal B}$ fixes exactly one point by the ATLAS of Finite Group Representations \cite{Atlas}, since the lengths of the suborbits are 1, 2, 2, 4, 4, and 8.  If $\fix_{{\mathcal A}^o}({\mathcal B}) = \PSL(2,7)$, then our claim follows, and otherwise the claim also follows as we would then have that $\fix_{{\mathcal A}^o}({\mathcal B})$ is isomorphic to one of $\PGL(2,7)$ or $\PGammaL(2,7)$, each of which is primitive in its action on $B\in{\mathcal B}$, by the ATLAS of Finite Group Representations \cite{Atlas}.

We may now apply \cref{maintool} and conclude that there exist $G_1,G_2 \le G$ such that $S=S_1\cup S_2$ $G_i=\langle S_i\rangle$, $G=G_1\times G_2$, and $\Gamma = \Gamma_1'\boxprod\Gamma_2'$, where $\Gamma_{i}'=\Cay(G_i,S_i)$ for $i=1,2$. By \cref{cartesian-CCA}, there must be some $i \in \{1,2\}$ such that $\Gamma_i$ is not a CCA graph on $G_i$. By induction, this $\Gamma_i$ contains a factor that is the unique non-CCA graph of order $21$, and so we may rearrange the Cartesian factors of $\Gamma$ to write $\Gamma = \Gamma_1\boxprod\Gamma_2$, where $\Gamma_2$ is the unique CCA-digraph of order $21$, while $\Gamma_1$ is a CCA graph of order $n/21$ (induction implies that it is a CCA graph, since its order is not divisible by 21).  Finally, by \cite[Corollary 15.6]{ImrichKR2008}, we have $\Aut(\Gamma) = \Aut(\Gamma_1)\times\Aut(\Gamma_2)$ (since $n$ is square-free, $\Gamma_1$ and $\Gamma_2$ cannot contain any common factors in their Cartesian decomposition). Applying \cref{maintool} again, with $\mathcal A\cong\PSL(2,7)$ being the color-preserving subgroup of $\Aut(\Gamma_2)\cong \PGL(2,7)$  by \cref{uniqueF21nonCCA}, whose action as noted above fixes precisely one point in each copy of $\Gamma_2$, we see that there exist $G_1, G_2 \le G$ with $G=G_1\times G_2$ and $\Gamma_2\cong\Cay(G_2,S_2)$. Since $F_{21}$ is the unique regular subgroup of $\Aut(\Gamma_2)=\PGL(2,7)$, it follows that $G_2\cong F_{21}$, so that $G=G_1\times F_{21}$, as desired.

It now only remains to consider the case where $\fix_{{\mathcal A}^o}({\mathcal B})$ is semiregular of order $p$. Then $\fix_{{\mathcal A}^o}({\mathcal B})=\fix_{G_L}(\mathcal B)\tl G_L$ by Lemma \ref{bottomsemiregular}.  If $G_L/{\mathcal B}\tl {\mathcal A}^o/{\mathcal B}$, then $G_L\tl{\mathcal A}^o$ and $\Gamma$ is CCA (see \cite[Remark 6.2]{HujdurovicKMM2016}).  We conclude that $G_L/{\mathcal B}$ is not normal in ${\mathcal A}^o/{\mathcal B}$ and ${\mathcal A}^o/{\mathcal B}$ is contained in the color-preserving group of automorphisms $A$ of $\Cay(G/N,S/N) = \Cay(G,S)/N$ by Lemma \ref{quotient color-preserving}.  Then $\Cay(G/N,S/N)$ is not CCA, and so by the induction hypothesis $\Cay(G/N,S/N) = \Gamma_1'\boxprod\Gamma_2'$ where $\Gamma_1'$ is a CCA-graph of a group $H$ of order $n/(21p)$ and $\Gamma_2'$ is the unique non-CCA graph of $F_{21}$.  Then $A$ admits an invariant partition ${\mathcal C}/{\mathcal B}$ consisting of blocks of size $21$.  As ${\mathcal A}^o/{\mathcal B}\le A$, this then implies that ${\mathcal A}^o$ admits an invariant partition ${\mathcal C}$ with blocks of size $21p$.

Now, by the induction hypothesis $G/N = H\times F_{21}$ and so $\Aut(G/N,S/N)\le \Aut(H)\times\Aut(F_{21})$.  As $\Gamma_1'$ is CCA with respect to $H$, any color-preserving automorphism of $\Cay(G/N,S/N)$ that is not a group automorphism of $H\times F_{21}$ cannot normalize $F_{21}$, viewed as an internal subgroup of $H\times F_{21}$. Since $\Cay(G/N,S/N)$ is not CCA, there exists some color-preserving automorphism $\alpha/\mathcal B$ of $\Cay(G/N,S/N)$ such that $\alpha/\mathcal B$ is not a group automorphism of $H \times F_{21}$, where $\alpha \in \mathcal A^o$. Thus we see that $\alpha/{\mathcal B}$ cannot normalize $F_{21}$, again viewed as an internal subgroup of $H\times F_{21}$.  As $F_{21}$ is a maximal subgroup of $\PSL(2,7)$ which is the color-preserving group of automorphisms of $\Gamma_2'$ by \cref{uniqueF21nonCCA}, we see $\fix_{{\mathcal A}^o/{\mathcal B}}({\mathcal C}/{\mathcal B}) = \PSL(2,7)$.

Suppose that $\fix_{{\mathcal A}^o}({\mathcal C})$ is not faithful in its action on $C\in{\mathcal C}$, with $K$ the kernel of this action.  If $\gamma\in K$, then $\gamma$ fixes a point, and so $K$ is a $2$-group by \cite[Lemma 6.3]{HujdurovicKMM2016}.  Let $C'\in{\mathcal C}$ such that $K\vert_{C'}\not = 1$.  Then $(K/{\mathcal B})\vert_{C'/{\mathcal B}}\not = 1$, and as $K/{\mathcal B}\tl\fix_{{\mathcal A}^o}({\mathcal C})/{\mathcal B}$, we see that $(K/{\mathcal B})\vert_{C'/{\mathcal B}} = \PSL(2,7)$, which is not a $2$-group, a contradiction.  Thus $\fix_{{\mathcal A}^o}({\mathcal C})$ is faithful in its action on $C\in{\mathcal C}$.

Now, as conjugation by an element of $\fix_{{\mathcal A}^o}({\mathcal C})$ induces an automorphism of $N$, there is a homomorphism $\phi:\fix_{{\mathcal A}^o}({\mathcal C})\mapsto\Aut(\Z_p)$, and of course $\Aut(\Z_p)$ is cyclic of order $p - 1$.  Then $N\le\Ker(\phi)$ and as $\fix_{{\mathcal A}^o}({\mathcal C})/N\cong \fix_{{\mathcal A}^o/{\mathcal B}}({\mathcal C}/{\mathcal B}) = \PSL(2,7)$ is simple (and so in particular is not cyclic), we see that $\Ker(\phi) = \fix_{{\mathcal A}^o}({\mathcal C})$.  Then $N$ is central in $\fix_{{\mathcal A}^o}({\mathcal C})$.  This implies that $\fix_{{\mathcal A}^o}({\mathcal C})$ is a central extension of the perfect group $\PSL(2,7)$.  If the commutator subgroup $L$ of $\fix_{{\mathcal A}^o}({\mathcal C})$ is not equal to $\fix_{{\mathcal A}^o}({\mathcal C})$, then as $\fix_{{\mathcal A}^o}({\mathcal C})/{\mathcal B}\cong\PSL(2,7)$, we have that $L\cong\PSL(2,7)$.  Then $L$ is a normal subgroup of ${\mathcal A}^o$, and so ${{\mathcal A}^o}$ admits a normal invariant partition ${\mathcal D}\prec{\mathcal C}$ with blocks of size $21$.  Then ${\mathcal A}^o$ contains a minimal normal subgroup that is not a cyclic group of prime degree, and this case reduces to one considered above.  Otherwise, $L = \fix_{{\mathcal A}^o}({\mathcal C})$ is a perfect central extension of $\PSL(2,7)$, and as $\PSL(2,7)$ has Schur multiplier $2$, by \cite[Theorem 33.8]{Aschbacher1986} we have $p = 2$, a contradiction.
\end{proof}

\section*{Acknowledgements}

The work of A. H. is supported in part by the Slovenian Research Agency (research program P1-0285 and research projects N1-0032, and N1-0038). The work of K. K. is supported in part by the Slovenian Research Agency (research program P1-0285 and research projects N1-0032, N1-0038, J1-6720, and J1-6743),
in part by WoodWisdom-Net+, W$^3$B, and in part by NSFC project 11561021.
The work of J.M. is supported in part by a Discovery Grant of the Natural Science and Engineering Research Council of Canada.

\bibliography{References}{}
\bibliographystyle{amsplain}

\end{document}